\documentclass[12pt]{amsart}
\usepackage{amssymb,hyperref}
\usepackage{latexsym}
\usepackage{amsmath}

\theoremstyle{plain}
\newtheorem{theorem}{Theorem}[section]
\newtheorem{proposition}[theorem]{Proposition}
\newtheorem{lemma}[theorem]{Lemma}

\newtheorem*{thm}{Theorem~\ref{newth}}
\newtheorem*{thmm}{Theorem~\ref{goren2}}

\theoremstyle{definition}
\newtheorem{question}[theorem]{Question}
\newtheorem{definition}[theorem]{Definition}
\newtheorem{remark}[theorem]{Remark}
\numberwithin{equation}{section}

\newcommand{\hgt}{\operatorname{ht}}
\newcommand{\m}{\mathfrak m}

\begin{document}
\title{A less restrictive Brian\c{c}on-Skoda theorem with coefficients}
\author{Ian M. Aberbach and Aline Hosry}
\address{Mathematics Department \\
 	University of Missouri\\
	Columbia, MO 65211 USA}

\date{\today}
\maketitle
\begin{abstract}
The Brian\c{c}on-Skoda theorem in its many versions has been studied by algebraists for several decades. In this paper, 
under some assumptions on an F-rational local ring $(R,\m)$, and an ideal $I$ of $R$ of analytic spread $\ell$ and height $g < \ell$, we improve on two theorems by Aberbach and Huneke.  Let $J$ be a reduction of $I$. We first give results on when the integral closure of $I^\ell$ is contained in the product $J I_{\ell-1}$, where $I_{\ell-1}$ is the intersection of the primary components of $I$ of height $\leq \ell-1$.  In the case that $R$ is also Gorenstein, we give results on when the integral closure of $I^{\ell-1}$ is contained in $J$. \\
\end{abstract}

\section{Introduction}

In this paper, all rings are assumed to be  commutative and  Noetherian with identity.

The theorem of Brian\c{c}on and Skoda was first proved in an analytic setting. Namely, let $ O_n= \mathbb{C}\{z_1,\dots,z_n\}$ be the ring of convergent power series in $n$ variables. Let $f \in O_n$ be a non-unit (i.e. $f$ vanishes at the origin), and let $J(f)= (\frac{\partial f}{\partial z_1},\dots,\frac{\partial f}{\partial z_n})O_n$ be the Jacobian ideal of $f$. Then one can see that $f \in \overline{J(f)}$, the integral closure of $J(f)$,  and in particular there is an integer $k$ such that $f^k \in J(f)$. John Mather raised the following question: Does there exist an integer $k$ that works for all non-units $f$?

Brian\c{c}on and Skoda first answered this question affirmatively by proving that the $n^{th}$ power of $f$ lies in $J(f)$. This is an immediate result of the following theorem:

\begin{theorem} \textup{(\cite{BS})} Let $I \subseteq O_n$ be an ideal generated by $\ell$ elements. Then for all $w \geq 0$, 
$$ \overline{I^{\ell+w}} \subseteq I^{w+1}.$$
\end{theorem} 

Since $f \in \overline{J(f)}$, $f^n \in \overline{J(f)}^n \subseteq \overline{J(f)^n} \subseteq J(f)$, by applying the Brian\c{c}on and Skoda theorem for $I= J(f)$, which has at most $n$ generators (taking $w$ to be zero). Hence $f^n \in J(f)$ and Mather's question is answered.

Lipman and Sathaye proved that this purely algebraic result can be extended to arbitrary regular local rings as follows:
\begin{theorem} \textup{(\cite{LS})} Let $(R,\m)$ be a regular local ring and suppose that $I$ is an ideal of $R$ generated by $\ell$ elements. Then for all $w \geq 0$, 
$$ \overline{I^{\ell+w}} \subseteq I^{w+1}.$$
\end{theorem}

Lipman and Teissier were partially able to extend this theorem to pseudo-rational rings \cite{LT}, while Aberbach and Huneke were able to prove the theorem for F-rational rings and rings of F-rational type in the equicharacteristic case \cite{AH1}. 

Initially motivated by trying to understand the relationship between the Cohen-Macaulayness of the Rees ring of $I$, $R[It]$, and the associated graded ring of $I$, $\text{gr}_I R$, various authors (see, e.g., \cite{AH1},\cite{AH2},\cite{AH3},\cite{AHT},\cite{HH},\cite{L},\cite{Sw}) have studied the coefficients involved in the Brian\c con-Skoda Theorem.  More specifically, if $J = (a_1,\ldots, a_\ell)$ is a minimal reduction of $I$ and $z \in \overline{I^\ell}= \overline{J^\ell}$, then when we write $z = \sum_{i=1}^\ell r_ia_i$, we may ask where the coefficients $r_i$ lie.  For instance, can we say that they lie in $I$?  Heuristically, when $\hgt(I) < \ell$, there is reason to have this occur.
One  such result is Theorem 3.6 of  \cite{AH2}.   We are able to substantially reduce the necessary hypotheses needed in that paper.  Explicitly, we prove the following result (see the next section for the definition of $I_{\ell-1}$):

\begin{thm} 
Let $(R,\m)$ be an F-rational Cohen-Macaulay local ring and $I \subseteq R$ be an ideal of analytic spread $\ell$ and of height $g < \ell$. If $J$ is any reduction of $I$ then $\overline{I^\ell} \subseteq J I_{\ell-1}.$
\end{thm}

We are also interested in reducing the power $\ell$ of $I$ in $\overline{I^\ell} \subseteq J$. This is possible in Gorenstein rings, with the aid of local duality, using ideas first utilized in \cite{H}.  More precisely, we improve on  Theorem 4.1 of \cite{AH1}, due to Aberbach and Huneke,  to show: 

\begin{thmm}  Let $(R,\m)$ be an F-rational Gorenstein local ring of dimension $d$ and characteristic $p > 0$. Suppose that $I$ is an ideal of height $g$ and analytic spread $\ell > g$. Assume that $I= I_{\ell-1}$ and that $R/I$ has depth at least $d-\ell+1$. Then for any reduction $J$ of $I$, we have $\overline{I^{\ell-1}} \subseteq J.$\\
\end{thmm}

\section{Preliminary results}

In this section, we review some of the definitions and results that will be used in this paper.

Let $(R,\m)$ be a Noetherian local ring and let $I$ be an ideal of $R$. An ideal $J \subseteq I$ is a \textit{reduction} of $I$ if there exists an integer $n$ such that $J I^n = I^{n+1}$ \cite{NR}. The least such integer is the \textit{reduction number} of $I$ with respect to $J$. A reduction $J$ of $I$ is called a \textit{minimal reduction} if $J$ is minimal with respect to inclusion among reductions. When $(R,\m)$ is local with infinite residue field, every minimal reduction $J$ of $I$ has the same number of minimal generators. This number is called the \textit{analytic spread} of $I$, denoted by $\ell(I)$, and we always have the inequalities ht($I$) $\leq \ell(I) \leq $ dim $R$. 
The \textit{analytic deviation} of $I$, denoted by $ad(I)$, is the difference between the analytic spread of $I$ and the height of $I$, i.e. $ad(I)= \ell(I)-$ht($I$). 
We also define $I^{un}$ to be the intersection of the minimal primary components of the ideal $I$ (under this definition $I^{un}$ has no embedded components but may have components of different heights or dimension).

An element $x$ of $R$ is said to be in the \textit{integral closure} of $I$, denoted by $\overline{I}$, if $x$ satisfies an equation of the form $x^k+a_1x^{k-1}+\dots+a_k=0$ where $a_i \in I^i$ for $1\leq i \leq k$. If an ideal $J\subseteq I$ is a reduction, then $\overline{J}=\overline{I}$.

Let $R$ be a Noetherian ring of prime characteristic $p>0$ and let $q$ be a varying power of $p$. Denote by $R^o$ the complement of the union of the minimal primes of $R$ and let $I$ be an ideal of $R$. Define $I^{[q]}=(i^q; i \in I)$, the ideal generated by the $q^{th}$ powers of all the elements of $I$. The \textit{tight closure} of $I$ is the ideal $I^*=\{x \in R;\, \mbox{for some}\; c \in R^o, cx^q \in I^{[q]},\; \mbox{for}\, q>>0\}$. We always have that $I \subseteq I^*\subseteq \overline{I}$. If $I^*=I$ then the ideal $I$ is said to be \textit{tightly closed}. We say that elements $x_1,\dots, x_n$ of $R$ are \textit{parameters} if the height of the ideal generated by them is at least $n$ (we allow them to be the whole ring, in which case the height is said to be $\infty$). The ring $R$ is said to be \textit{F-rational} if the ideals generated by parameters are tightly closed.

The remaining ingredients of this section are from \cite{AH1} and \cite{AH2}.

\begin{definition} \label{basgen}\textup{(\cite{AH2}, Definition 2.10)} Let $R$ be a Noetherian local ring and let $I$ be an ideal of height $g$. We say that a reduction $J= (a_1,\dots,a_\ell)$ of $I$ is generated by a \textit{basic generating set} if for all prime ideals $P$ containing $I$ such that $i$= ht($P) \leq \ell$, $(a_1,\dots,a_i)_P$ is a reduction of $I_P$.
\end{definition}

When the residue field  of $R$ is infinite, there always exist such basic generating sets, and furthermore, ht$((a_1,\dots,a_i)I^n :I^{n+1} + I) \geq i+1$ for $n \gg 0$.

\begin{proposition} \label{mainprop1} \textup{(\cite{AH1}, Proposition 3.2)} Let $(R,\m)$ be an equidimensional and catenary local ring with infinite residue field and let $I\subseteq R$ be an ideal of height $g$ and analytic spread $\ell$. Let $J\subseteq I$ be a minimal reduction of $I$. We assume that  $J=(a'_1,\dots,a'_\ell)$ is generated by a basic generating set as in Definition \ref{basgen} above. Let $N$ and $w$ be fixed integers, and suppose that for $g+1 \leq i \leq \ell$, we are given finite sets of primes $\Lambda_i =\{Q_{ji}\}$ all containing $I$ and of height $i$. Then there exist elements $a_1,\dots,a_\ell$ and $t_{g+1},\dots,t_\ell$ such that the following conditions hold. (We set $t_i=0$ for $i\leq g$ for convenience).
\begin{itemize}
\item [1.] $a_i\equiv a'_i$ modulo $I^2$.
\item [2.] For $g+1 \leq i \leq \ell$, $t_i \in m^N$.
\item [3.] $b_1,\dots,b_g,b_{g+1},\dots,b_\ell$ are parameters, where $b_i=a_i+t_i.$
\item [4.] The images of $t_{g+1},\dots,t_\ell$ in $R/I$ are part of a system of parameters.
\item [5.] There is an integer $M$ such that $t_{i+1} \in (J_i^n I^M :I^{M+n})$ for all $0\leq n \leq w+\ell$ where $J_i= (a_1,\dots,a_i)$.
\item [6.] $t_{i+1} \notin \cup_j Q_{ji}$, the union being over the primes in $\Lambda_i$.
\end{itemize}
\end{proposition}

\begin{remark} We have altered the statement made in part 4 of Proposition~\ref{mainprop1} from that of the original, but the given statement holds.
\end{remark}

Proposition~\ref{mainprop1} allows us to choose a parameter ideal, $\m$-adically as close to $I$ as desired, for which a sort of Brian\c con-Skoda result applies.  Specifically:

\begin{theorem}\label{inclusion} \textup{(\cite{AH1}, Theorem 3.3)} Let $(R,\m)$ be an equidimensional and catenary local ring of characteristic $p$ having an infinite residue field. Let $I$ be an ideal of analytic spread $\ell$ and positive height $g$. Let $J$ be a minimal reduction of $I$. Fix $w$ and  $N \geq 0$. Choose $a_i$ and $t_i$ as in Proposition \ref{mainprop1}. Set $\mathfrak{A}=B_\ell=(b_1,\dots,b_g,\dots,b_\ell).$ Then 
$$ \overline{I^{\ell+w}} \subseteq (\mathfrak{A}^{w+1})^*.$$ 
\end{theorem}




One of our main goals in this paper is to generalize the next theorem, which is due to Aberbach and Huneke. We recall that an ideal $I$ satisfies a property generically if $I_p$ satisfies that same property for every minimal prime $P$ of $I$.

\begin{theorem}\label{oldth} \textup{(\cite{AH2}, Theorem 3.6)}
Let $(R,\m)$ be an excellent F-rational local ring, $I\subseteq R$ an ideal with reduction $J$.  Let $g = \text{ht}(I) <\ell = \ell(I)$.  Suppose that 
\begin{itemize}
\item $R/I$ is equidimensional. 
\item $R/I^{un}$ satisfies $S_{\ell-g-1}$, and
\item $I$ is generically of reduction number at most one. 
\end{itemize}
Then $\overline{ I^{\ell}} \subseteq JI^{un}.$ 

In particular if $R/I$ is equidimensional, $I$ generically has reduction number at most one and $I$ has analytic deviation 2, then $\overline{I^\ell} \subseteq JI^{un}$.\\
\end{theorem}

\section{The main theorem}

 We will show that Serre's condition $S_{\ell-g-1}$ on $R/I^{un}$, and the assumption that $I$ is an ideal generically of reduction number at most one in Theorem \ref{oldth} of Aberbach and Huneke are not necessary.  We will also show that $I^{un}$ may be replaced by a (potentially) smaller ideal, which in many instances, also allows us to remove the hypothesis that $I$ is unmixed.

\begin{lemma} \label{stab} Let $R$ be a Noetherian ring, $J$ an ideal of $R$, and $x$ an element of $R$.  Then there exists a positive integer M such that $J:_R x^{\infty} = J:_R x^M.$
\end {lemma}

\begin{proof}[{\bf Proof}]
$J:_R x \subseteq J:_R x^2 \subseteq \dots$ is an increasing sequence of ideals in the Noetherian ring $R$, so there exists $M$ such that 
$J:_R x^{M+k} = J:_R x^M \; \mbox {for all}\; k  \geq 0.$
Hence $J:_R x^{\infty} = \bigcup_ {i \in \mathbb {N}} (J:_R x^i) = J :_R x^M.$ 
\end{proof}

\begin{remark} \label{highpow} After renaming $x^M$ back to $x$, we can assume that $J:_R x^{\infty} = J:_R x.$
\end{remark}

\begin{proposition} \label{new} Let $R$ be a Noetherian ring and $I$ an ideal of $R$. Let $t_1$ be an element of $R$ such that its image is regular in $R/I$. Then for any elements $t_2,\dots,t_n$ of $R$, there exist elements $u_2,\dots,u_n$ of $R$ such that for all $2\leq i \leq n$, $u_i$ is a power of $t_i$, and $(I,u_2^2,\dots,u_n^2): t_1 \subseteq (I,u_2,\dots,u_n).$
\end{proposition}

\begin{proof}[{\bf Proof}] Let $s_1$ be the image of $t_1$ in $S=R/I$ and set $S_1 = S/ (s_1).$

By Lemma \ref{stab} and Remark \ref{highpow}, we can pick $u_2 \in R$, a power of $t_2$, such that its image $s_2$ in $S$ satisfies: $0:_{S_1} s_2^{\infty} = 0:_{S_1} s_2.$ 
For $3 \leq i \leq n$, pick $u_i \in R$, a power of $t_i$, such that its image $s_i$ in $S$ satisfies $$(s_2^2,\dots,s_{i-1}^2):_{S_1} s_i^{\infty} = (s_2^2,\dots, s_{i-1}^2):_{S_1} s_i.$$

We claim that for $2 \leq i \leq n$, if $w \in (s_2^2,\dots,s_i^2):_S s_1$, there exists $v_i$ such that $$w-v_i s_i \in (s_2^2,\dots,s_{i-1}^2):_S s_1.$$
To prove the claim, let $w$ be in $(s_2^2,\dots,s_i^2):_S s_1$, then one can write that 
\begin{equation} \label{eq1}
s_1 w = \alpha_2 s_2^2+\dots+ \alpha_i s_i^2.
\end{equation}
Thus $ \alpha_2 s_2^2+\dots+ \alpha_i s_i^2 = 0$ in $S_1$, which implies that $$\alpha_i \in (s_2^2,\dots, s_{i-1}^2) : _{S_1} s_i^2 \subseteq (s_2^2,\dots, s_{i-1}^2) :_ {S_1} s_i^{\infty} = (s_2^2,\dots, s_{i-1}^2) :_ {S_1} s_i.$$ 
Hence $\alpha_i \in (s_2^2,\dots, s_{i-1}^2):_{S_1} s_i$, or equivalently $\alpha_i s_i \in (s_1, s_2^2,\dots,s_{i-1}^2)S.$

Write $\alpha_i s_i = v_i s_1+ x_i$ where $x_i \in (s_2^2,\dots, s_{i-1}^2).$
We get $\alpha_i s_i^2 = v_i s_1 s_i + x_i s_i$. Replacing this expression for $\alpha_i s_i^2$ back into (\ref{eq1}), and combining the terms involving $s_1$, we see that
$ s_1(w- v_i s_i) \in (s_2^2,\dots,s_{i-1}^2)$, which implies that $w- v_i s_i \in (s_2^2,\dots,s_{i-1}^2):_S s_1$, as desired.

Therefore, we can conclude that if $w \in(s_2^2,\dots,s_n^2):_S s_1$, there exist $v_2,\dots, v_n$ such that  
$s_1(w- v_n s_n-\dots- v_2 s_2) = 0 \;\mbox{in}\; S.$

But $s_1$ is a nonzerodivisor in $S$, so $w-v_n s_n- \dots- v_2 s_2 = 0$ in $S$, implying that $w \in (s_2,\dots,s_n)S$.   Therefore $(s_2^2,\dots,s_n^2) :_S s_1 \subseteq (s_2,\dots,s_n)S$, or equivalently, $(I,u_2^2,\dots,u_n^2): t_1 \subseteq (I,u_2,\dots,u_n)R$ as desired. 
\end{proof}

\begin{definition} If $I$ has height $g$, then given $k \geq g$, set $S_k= R \setminus \bigcup P$ where the union is taken over all primes $P$ that are associated to $I$ and such that ht ($ P) \leq k$. We define $I_k= IS_k^{-1}R \cap R$. This means that given a primary decomposition of $I$, $I_k$ is the intersection of the primary components of $I$ of height $\leq k$.
\end{definition}

Let $(R,\m)$ be an equidimensional and catenary local ring with infinite residue field and let $I \subseteq R$ be an ideal of height $g$ and analytic spread $\ell$. Let $J \subseteq I$ be a minimal reduction of $I$. We assume that $J= (a_1',\dots,a_\ell')$ is generated by a basic generating set as in Definition \ref{basgen}. Let $N$ and $w$ be fixed integers, and suppose that for $g+1 \leq i \leq \ell$ we are given finite sets of primes $\Lambda_i = \{Q_{ji}\}$ all containing $I$ and of height $i$.

\begin{proposition} \label{mainprop2} With the above assumptions, there exist elements $a_1,\dots,a_\ell$ generating $J$  and $t_{g+1},\dots,t_\ell$ of $R$ such that conditions 1 through 6 of Proposition \ref{mainprop1} hold.

In addition, we can choose the elements $t_{g+1},\dots, t_\ell$ in $R$ such that 
$$(I_{\ell-1}, t_{g+1}^2,\dots,t_{\ell-1}^2) :_R t_{\ell} \subseteq (I_{\ell-1}, t_{g+1},\dots,t_{\ell-1}).$$
\end{proposition}

\begin{proof}[{\bf Proof}] Pick elements $a_1,\dots,a_\ell$ and $t_{g+1},\dots, t_\ell$ in $R$ as in Proposition \ref{mainprop1}. If necessary, replace $t_{g+1},\dots,t_{\ell-1}$ by higher powers so that $$(I_{\ell-1}, t_{g+1}^2,\dots,t_{\ell-1}^2) :_R t_{\ell} \subseteq (I_{\ell-1}, t_{g+1},\dots,t_{\ell-1}).$$

This is possible since on one hand properties 1-6 of Proposition \ref{mainprop1} remain true after replacing $t_{g+1},\dots,t_{\ell-1}$ by higher powers. And on the other hand, we can apply Proposition \ref{new} once we check that $t_{\ell}$ is a nonzerodivisor in $R/I_{\ell-1}$. But this is true by property 6 of Proposition \ref{mainprop1} if we set $\Lambda_{\ell-1}$ to be any finite set of height $\ell-1$ primes whose union contains all associated primes of $I$ with height at most $\ell-1$.
\end{proof}

A few more results are needed before we can give a proof to our main theorem in this section.

Let $(R,\m)$ be an equidimensional and catenary local ring of characteristic $p$, having an infinite residue field. Let $I$ be an ideal of analytic spread $\ell$ and positive height $g$. Let $J$ be a minimal reduction of $I$. Fix integers $w$ and $N \geq 0$, and choose $a_1,\dots,a_\ell$ and $t_{g+1},\dots,t_\ell$ as in Proposition \ref{mainprop1}. For $i=1,\dots,\ell$, set $b_i= a_i+t_i$ ($t_i=0$ for $i\leq g$), $J_i=(a_1,\dots,a_i)$, and $B_i = (b_1, \ldots, b_i)$.

\begin{lemma} \label{facts} With the above assumptions, there exists an element $c \in I^M \cap R^0$ (where $M$ is the integer from condition 5 of Proposition \ref{mainprop1}) such that the following conditions hold:
\begin{enumerate}
\item [1.] For any $g+1 \leq j \leq \ell$ and $1\leq k \leq w+\ell$, $c t_j^q \, I^{kq} \subseteq J_{j-1}^{k q}$, for any power $q$ of $p$.
\item [2.] For all $g \leq i \leq \ell$ and $0 \leq r \leq w$, we have $c^{i-g} \, J_i^{(i+r)q} \subseteq (B_i^{r+1})^{[q]}$, for any power $q$ of $p$.
\end{enumerate}
\end{lemma}

\begin{proof} [{\bf Proof}] This lemma combines useful facts that were presented in the proof of Theorem \ref{inclusion}. For their proofs, refer to the proof of Theorem 3.3 in \cite{AH1}.
\end{proof}

The next result is a generalization of Lemma 4.3 in \cite{AH1}.

\begin{lemma}\label{mine} Under the above assumptions, assume that $g < \ell$. Let $m$ be $0$ or $-1$. Then for all $g+1 \leq j \leq \ell$, we have $$t_j \, \overline{I^{\ell+m}} \subseteq (B_{j-1}^{\ell-j+m+2})^*.$$
\end{lemma}

\begin{proof}[{\bf Proof}]  Fix an integer $w \geq \ell-g  \geq 0$ and let $z \in \overline{I^{\ell+m}}$. Then there exists an element $d \in R^0$ such that $d z^q \in I^{(\ell+m)q}$, for all powers $q$ of $p$.  
Also, choose $c \in I^M \cap R^0$ satisfying the conclusions of Lemma \ref{facts}. 

For $g+1 \leq j \leq \ell$, $ d c^{j-g} t_j^q z^q \in t_j^q c^{j-g}\, I^{(\ell+m)q} = c^{j-g-1}c t_j^q \,I^{(\ell+m)q}$.
Apply  Lemma \ref{facts}(1) and (2)  to obtain $d c^{j-g} t_j^q z^q \in c^{j-1-g}\, J_{j-1}^{(\ell+m)q} \subseteq (B_{j-1}^{\ell-j+m+2})^{[q]}$ when taking $g \leq i=j-1 \leq \ell$ and $0\leq r=\ell-j+m+1 \leq \ell -g \leq w$.

Hence, $d c^{j-g} t_j^q z^q \in (B_{j-1}^{\ell-j+m+2})^{[q]}$ which implies that $t_j z \in (B_{j-1}^{\ell-j+m+2})^*$, since the element $ d c^{j-g}$ is in $R^0$. 
Therefore, we conclude that $t_j \,\overline{I^{\ell+m}} \subseteq (B_{j-1}^{\ell-j+m+2})^*$, for $m=0$ or $-1$.
\end{proof}

\begin{remark} In Theorem \ref{inclusion} and Lemmas \ref{facts} and \ref{mine}, if one replaces any $b_i=a_i+t_i$ by $a_i+t_i^2$, the conclusions remain unchanged. This is true because by raising any $t_i$ to a higher power, conditions 1 through 6 of Proposition \ref{mainprop1} still hold.
\end{remark}

The next theorem generalizes  Aberbach and Huneke's Theorem 3.6 of \cite{AH1}, stated here as Theorem~\ref{oldth}. It shows that Serre's condition, the assumption that $R/I$ is equidimensional and the generic reduction number hypothesis are superfluous. In the proof that we present, we make the appropriate modifications to Aberbach and Huneke's proof of Theorem \ref{oldth}. 

\begin{theorem} \label{newth} Let $(R,\m)$ be an F-rational Cohen-Macaulay local ring (e.g., an excellent F-rational local ring), $I \subseteq R$ an ideal of analytic spread $\ell$ and of height $g < \ell$ and let $J$ be any reduction of $I$. Then $\overline{I^\ell} \subseteq J I_{\ell-1}.$
\end{theorem}

\begin{proof}[{\bf Proof}] Without loss of generality, assume that $R$ has an infinite residue field and $J= (a_1,\dots,a_\ell)$ is a minimal reduction generated by a basic generating set.

Fix an integer $N$ and choose $t_{g+1},\dots,t_\ell \in m^N$ as in Proposition \ref{mainprop2} (here we set $w=0$). Hence, $t_{g+1},\dots,t_\ell$ satisfy the conditions of Proposition \ref{mainprop1} as well as the inclusion  $$(I_{\ell-1}, t_{g+1}^2,\dots,t_{\ell-1}^2) :_R t_{\ell} \subseteq (I_{\ell-1}, t_{g+1},\dots,t_{\ell-1}).$$

Set $b_k= a_k + t_k^2$ (we set $t_k=0$ for $k \leq g$). 
By Theorem \ref{inclusion}, $\overline{I^\ell} \subseteq (a_1,\dots, a_g,b_{g+1},\dots,b_\ell).$
Given $z \in \overline{I^\ell}$, we may write $z= r_1 a_1+ \dots+r_g a_g+ r_{g+1}b_{g+1}+\dots+r_\ell b_\ell$, where $r_i \in R$, for $1\leq i \leq \ell.$
We aim to show that $r_i \in I_{\ell -1} + m^N$, for all $i=1,\dots,\ell$.



For $1 \leq i \leq \ell$, 
\begin{equation*}
\begin{split}
t_{\ell} r_i b_i & \in t_\ell\,(\overline{I^\ell}, b_1,\dots,\widehat{b_i},\dots,b_{\ell})\\
                 & \subseteq ((a_1,\dots,a_g,b_{g+1},\dots,b_{\ell-1})^2)^*+ (b_1,\dots,\widehat{b_i},\dots,b_{\ell}), \; \mbox{by Lemma \ref{mine}} \\
                & \subseteq (b_1,\dots,b_i^2,\dots,b_\ell).  
\end{split}
\end{equation*}
By combining the terms involving $b_i$, we conclude that
\begin{equation*}
\begin{split}
t_{\ell} r_i & \in (b_i)+ (b_1,\dots,\widehat{b_i},\dots,b_\ell):b_i \\
            & = (b_i)+(b_1,\dots,\widehat{b_i},\dots,b_\ell), \; \mbox{since $b_1,\dots,b_{\ell}$ is a regular sequence} \\
            & \subseteq (J, t_{g+1}^2,\dots,t_{\ell}^2) \\
            & \subseteq (I_{\ell-1}, t_{g+1}^2,\dots,t_{\ell}^2).
\end{split}
\end{equation*}

Now combine the terms containing $t_{\ell}$ to obtain that  
\begin{equation*}
\begin{split}
r_i & \in (t_{\ell})+ (I_{\ell-1},t_{g+1}^2,\dots,t_{\ell-1}^2): t_{\ell}\\ 
    & \subseteq (t_{\ell})+(I_{\ell-1},t_{g+1},\dots,t_{\ell-1}), \; \mbox{by Proposition \ref{mainprop2}.}\\
\end{split}
\end{equation*}
Hence, $r_i \in (I_{\ell-1},t_{g+1},\dots,t_\ell) \subseteq I_{\ell-1} + m^N$, for all $i=1,\dots,\ell$.

We conclude that $z \in I_{\ell-1}(a_1,\dots,a_g)+ I_{\ell-1}(b_{g+1},\dots,b_\ell)+m^N \subseteq I_{\ell-1}(a_1,\dots,a_\ell) +m^N$.

As $N$ was arbitrary, the Krull intersection theorem gives that $z \in JI_{\ell-1}$, finishing the proof of the theorem.
\end{proof}

\begin{remark} When $R/I$ is equidimensional, $I_{\ell-1} \subseteq I^{un}$, so Theorem \ref{newth} implies in this case that $\overline{I^{\ell}} \subseteq J I^{un}$. Hence Theorem~\ref{newth} is a generalization of Aberbach and Huneke's Theorem \ref{oldth}, but removes the hypotheses involving Serre's condition on $R/I^{un}$, the generic reduction number of $I$, and the assumption that $R/I$ is equidimensional.\\
\end{remark}

\section{A theorem for F-rational Gorenstein rings}

In this section, we are interested in the cases where the power $\ell$ of $I$ in the inclusion $\overline{I^\ell} \subseteq J$ (where $J$ is a reduction of $I$), can be lowered. A cancellation theorem due to Huneke \cite{H} inspired the main idea behind the proof of our next result. In particular, we extend another theorem of Aberbach and Huneke that states:

\begin{theorem} \label{goren1}\textup{(\cite{AH1}, Theorem 4.1)} Let $(R,\m)$ be an F-rational Gorenstein local ring of dimension $d$ and having positive characteristic. Suppose that $I$ is an ideal of height $g$ and analytic spread $\ell >g$, with $R/I$ Cohen-Macaulay. Then for any reduction $J$ of $I$, $\overline{I^{\ell-1}} \subseteq J$.
\end{theorem}

We extend Theorem \ref{goren1} of Aberbach and Huneke in the following way:

\begin{theorem} \label{goren2} Let $(R,m)$ be an F-rational Gorenstein local ring of dimension $d$ and characteristic $p > 0$. Suppose that $I$ is an ideal of height $g$ and analytic spread $\ell > g$. Assume that $I= I_{\ell-1}$ and that $R/I$ has depth at least $d-\ell+1$. Then for any reduction $J$ of $I$, we have $\overline{I^{\ell-1}} \subseteq J.$

In particular, if $\ell =d$ and $I=I_{\ell-1}$, then $\overline{I^{\ell-1}} \subseteq J$.
\end{theorem}

\begin{proof} [{\bf Proof}] The proof is a modification of the proof of Theorem \ref{goren1} presented in \cite{AH1}.

Without loss of generality, we may assume that $R$ has an infinite residue field and that $J$ is a minimal reduction of $I$. Fix an integer $N \geq 0$, and set $\Lambda_{\ell-1}$ to be any finite set of primes of $R$ of height $\ell-1$ such the the union contains all associated primes of $I$ of height at most $\ell-1$. Choose $a_1,\dots,a_\ell$ and $t_{g+1},\dots,t_\ell$ as in Proposition \ref{mainprop1} (here we set $w=0$). 

For $1 \leq i \leq \ell$, let $b_i=a_i+t^2_i$ (with $t_i=0$ for $i\leq g$), $J_i= (a_1,\dots,a_i)$ and $B_i = (b_1,\dots, b_i)$. 

By our choice of $\Lambda_{\ell-1}$, $t_{\ell}$ is a nonzerodivisor in $R/I_{\ell-1}= R/I$. Since depth $R/I \geq d-\ell +1$, we can pick elements $x_{\ell+1},\dots,x_d$ in $R$ such that $b_1,\dots,b_\ell,x_{\ell+1},\dots,x_d$ is a regular sequence in $R$, and such that $t_{\ell},x_{\ell+1},\dots,x_d$ is a regular sequence in $R/I$. 

By Proposition \ref{new}  we can replace $t_{g+1},\dots,t_{\ell-1}$ by higher powers of themselves so that
\begin{equation} \label{eq2}
(I, t_{g+1}^2,\dots,t_{\ell-1}^2,x_{\ell+1},\dots,x_d):t_{\ell}^2 \subseteq (I,t_{g+1},\dots,t_{\ell-1},x_{\ell+1},\dots,x_d),
\end{equation}
where to obtain this inclusion we use that $t_{\ell}^2$ is regular modulo $(I,x_{\ell+1},\dots,x_d)$.

Set $\mathfrak{A}=B_\ell +(x_{\ell+1},\dots,x_d)$, $D=B_{\ell-1}: t^2_{\ell}$ and $K=B_{\ell-1} + (x_{\ell+1},\dots,x_d)$.   Note that $K:b_\ell = K$ since the elements involved form a regular sequence in $R$.

Let $Q=(I,t_{g+1},\dots,t_{\ell-1},x_{\ell+1},\dots,x_d)+ K:D$. We claim that $\mathfrak{A} : t^2_{\ell} \subseteq Q.$ Let $t^2_{\ell}u \in \mathfrak{A}$ and write 
\begin{equation} \label{eq3}
t^2_{\ell}u=w+v b_{\ell}= w+ v a_{\ell}+ v t^2_{\ell},
\end{equation}
where $w \in K$. Then $t_{\ell}^2(u-v) \in B_{\ell-1}+(a_{\ell},x_{\ell+1},\dots,x_d),$ and hence
\begin{equation*}
\begin{split}
u-v  \in (B_{\ell-1}+(a_{\ell},x_{\ell+1},\dots,x_d)):t^2_{\ell}
    & \subseteq (I,t_{g+1}^2,\dots,t_{\ell-1}^2,x_{\ell+1},\dots,x_d):t^2_{\ell} \\
    & \subseteq (I,t_{g+1},\dots,t_{\ell-1},x_{\ell+1},\dots,x_d), \mbox{by (\ref{eq2})}.
\end{split}
\end{equation*}

Thus, $u-v \in Q.$ To show that $u \in Q$, it suffices to show that $v \in K: D \subseteq Q$. Let $d \in D$ and consider $dv$. By (\ref{eq3}), $dt^2_{\ell}u=dw+dvb_{\ell}$. But as $dt^2_{\ell} \in B_{\ell-1}$, $dvb_{\ell} \in K$. Therefore, $dv \in K:b_\ell = K$, as noted above.
Consequently, $dv \in K$ and $v \in K:D \subseteq Q$. Hence, $u \in Q$ and this proves the claim  that $\mathfrak{A}:t^2_{\ell} \subseteq Q$. In particular it proves that $\mathfrak{A}: Q \subseteq \mathfrak{A}: (\mathfrak{A}:t^2_{\ell}).$

Next, we show that $\overline{I^{\ell-1}} \subseteq \mathfrak{A}:Q$. 
First, recall that $t^2_{\ell}\,\overline{I^{\ell-1}} \subseteq B_{\ell-1}$, by Lemma \ref{mine}. Thus, $\overline{I^{\ell-1}} \subseteq D$, and hence $\overline{I^{\ell-1}}(K:D) \subseteq D(K:D) \subseteq K \subseteq \mathfrak{A}$.
Moreover, $I \,\overline{I^{\ell-1}} \subseteq \overline{I^{\ell}} \subseteq \mathfrak{A}^* = \mathfrak{A}$, by Theorem \ref{inclusion} and the fact that $R$ is F-rational.

For $g+1 \leq j \leq \ell-1$, Lemma \ref{mine} implies that 
\begin{equation*}
t_j\, \overline{I^{\ell-1}}  \subseteq ((B_{j-1}^2)^{\ell-j+1})^* 
                              \subseteq B_{j-1}^* 
                              = B_{j-1} 
                              \subseteq \mathfrak{A}.
\end{equation*}
Consequently, $(t_{g+1},\dots,t_{\ell-1})\,\overline{I^{\ell-1}} \subseteq \mathfrak{A}.$
Therefore, we have proved that $\overline{I^{\ell-1}} \subseteq \mathfrak{A}:Q.$

Finally, $\overline{I^{\ell-1}} \subseteq \mathfrak{A}:Q \subseteq \mathfrak{A}:(\mathfrak{A}:t^2_{\ell}) = (\mathfrak{A}, t^2_{\ell})$, by local duality. Hence, $$\overline{I^{\ell-1}} \subseteq (J,t_{g+1}^2,\dots,t_{\ell-1}^2,t^2_{\ell},x_{\ell+1},\dots,x_d) \subseteq J+m^N.$$
An application of the Krull intersection theorem proves that $\overline{I^{\ell-1}} \subseteq J$.
\end{proof} 

\begin{remark} 
If $I$ is unmixed and equidimensional (e.g., $R/I$ is CM) in an F-rational Gorenstein ring then
 $I=I_{\ell-1}$. Therefore, Theorem~\ref{goren2} is a generalization of Theorem~\ref{goren1}.
 \end{remark}

\begin{question}  Suppose that $\ell -g \ge 2$.  Is it possible to improve on Theorem~\ref{goren2} to obtain that $\overline{I^{\ell-2}} \subseteq J$?  If, in the proof of Theorem~\ref{goren2}, we could show that, in fact, $\overline{I^{\ell-1}} \subseteq B_\ell + (x_{\ell+1},\dots, x_d)$, then one could extend the result.\\
\end{question}

\end{document}